\documentclass[11pt]{amsart}
\usepackage{amssymb,amsmath,amsthm}
\usepackage{tikz}
\usepackage{enumitem}
 \usepackage{tikz-cd}
 \usepackage{mathtools}
\usetikzlibrary{positioning}
\usetikzlibrary{matrix,arrows,decorations.pathmorphing, shapes.geometric}
\tikzset{>=latex}
\usepackage{verbatim}

\usepackage{xcolor}

\newtheorem{theorem}[equation]{Theorem}
\newtheorem*{theorem*}{Theorem}
\newtheorem{lemma}[equation]{Lemma}
\newtheorem*{lemma*}{Lemma}
\newtheorem{prop}[equation]{Proposition}

\newtheorem*{corollary*}{Corollary}

\theoremstyle{remark}
\newtheorem{remark}[equation]{Remark}

\numberwithin{equation}{section}

\newcommand{\Sph}{\mathbb{S}}
\newcommand{\Rcap}{\mathsf{R}}

\newcommand{\N}{\mathbb{N}}

\newcommand{\R}{\mathbb{R}}

\newcommand{\C}{\mathbb{C}}
\newcommand{\HH}{\mathbb{H}}
\newcommand{\ddiv}{\mathrm{div}}

\newcommand{\sn}{\operatorname{sn}}
\newcommand{\ct}{\operatorname{ct}}
\newcommand{\cs}{\operatorname{cs}}

\newcommand{\Ecal}{\mathcal{E}}

\newcommand{\Kcal}{\text{\Tiny $\mathcal{K}$}}
\newcommand{\Kmin}{\underline{K}}
\newcommand{\Kmax}{\overline{K}}

\begin{document}
\title[Symmetry and Isoperimetry]{Symmetry and Isoperimetry for Riemannian Surfaces}


\author[J.~Hoisington]{Joseph~Ansel~Hoisington}
\author[P.~McGrath]{Peter~McGrath}
\date{}
\address{Department of Mathematics, University of Georgia, Athens GA 30602}
\address{Department of Mathematics, North Carolina State University, Raleigh NC} 
\email{jhoisington@uga.edu}
\email{pjmcgrat@ncsu.edu}
\maketitle

\begin{abstract}

For a domain $\Omega$ in a geodesically convex surface, we introduce a scattering energy $\Ecal(\Omega)$, which measures the asymmetry of $\Omega$ by quantifying its incompatibility with an isometric circle action.  We prove several sharp quantitative isoperimetric inequalities involving $\Ecal(\Omega)$ and characterize the domains with vanishing scattering energy by their convexity and rotational symmetry.  We also give a new of the sharp Sobolev inequality for Riemannian surfaces. 
\end{abstract}

\section{Introduction}
\label{intro}
\nopagebreak

In this article, we study symmetry, isoperimetric inequalities and Sobolev inequalities on surfaces.  To begin, let $(\Sigma, g)$ be a Riemannian surface, possibly with boundary, in which each pair of points is joined by a unique, minimizing geodesic.  For distinct $x, y\in \Sigma$,  define a linear map $\Rcap : T_y \Sigma \rightarrow T_x \Sigma$ as follows:  for $v \in T_y\Sigma$, reflect $v$ in $T_y\Sigma$ across the axis orthogonal to the geodesic segment from $y$ to $x$ and define $\Rcap v \in T_x \Sigma$ to be the parallel translate of the result along this segment (see Figure 1).  Let $\Omega \subset \Sigma$ be a precompact domain with $C^2$ boundary.  We define the \emph{scattering energy} $\Ecal(\Omega)$ associated to $\Omega$ by
\begin{align}
\label{EAindex}
\Ecal(\Omega) := \frac{1}{2}\iint_{\partial \Omega \times \partial \Omega}\! | \nu_x- \Rcap \nu_y |^2 \, ds_x ds_y,
\end{align} 
where $\nu_p$ denotes the outward unit normal to $\Omega$ at $p$, $ds_x$ and $ds_y$ are arclength elements along $\partial \Omega$, and by convention we define $(\Rcap \nu_y)(y) = \nu_y$.

\begin{theorem}
\label{Ciso}
For $\Omega$ as above and $\Ecal(\Omega)$ as defined in \eqref{EAindex},
	\begin{equation}
	\label{Eiso}
	L^2 - 4\pi A + ( \sup_{\Omega'} K ) A^2  \geq \Ecal(\Omega),
	\end{equation}
	where $L$ and $A$ are the boundary length and area of $\Omega$, $\Omega'$ is the union of the geodesic segments in $\Sigma$ joining points of $\Omega$, $K$ is the curvature of $\Sigma$, and if $\sup_{\Omega'} K > 0$, then we assume that $diam(\Omega) \leq \pi/ (2\sqrt{\sup_{\Omega'}K})$.  Equality holds if and only if $\Omega'$ has constant curvature. 
	\end{theorem}
We will refer to the union of the geodesic segments $\Omega'$ as the {\em geodesic hull} of $\Omega$ in $\Sigma$.  Theorem \ref{Ciso} is a strong form of an isoperimetric inequality due to Bol \cite{Bol}, which extended results of Weil \cite{Weil} and Beckenbach-Rado \cite{Rado} for surfaces of nonpositive curvature.  We refer to \cite{Croke, Kleiner1992, Prince} for background on isoperimetric inequalities in manifolds with upper curvature bounds and \cite{Topping} for more recent results on the isoperimetric inequality on surfaces.  

In Theorem \ref{Trigidity} we characterize the domains with vanishing scattering energy.  We note first that if $\Omega$ is convex and $O(2)$-symmetric, then $\Ecal(\Omega) =0$:  for each pair $x, y \in \partial \Omega$, there is a unique reflection---that is, an orientation-reversing isometry of $\Omega$---exchanging $x$ and $y$.  This reflection maps the geodesic segment between $x$ and $y$ to itself, reversing its orientation, which implies that $\Rcap \nu_y = \nu_x$, and thus that $\Ecal(\Omega) =0$.  The converse also holds: 

\begin{theorem}
\label{Trigidity}
	Suppose that $\Ecal(\Omega) = 0$ and that no pairs of boundary points of $\Omega$ are conjugate.  Then $\Omega$ is strictly convex and admits an isometric action by the orthogonal group $O(2)$.
\end{theorem}

\begin{figure}[h]
	\centering
	\begin{tikzpicture}
	\draw[-, thick, black] (0, 0) to (4, 0);
	\draw[->, thick] (0, 0) node[below] {$y$} -- (-2/3, 2/3) node[left] {$v$};
	\draw[->, thick] (4, 0) node[below] {$x$} -- (4+2/3, 2/3) node[right] {$\Rcap v$};
	\draw[dashed, thin, black] (0, 0) to (0, 1);	
	\draw[->, black, thick] (0, 0) to (2/3, 2/3);	
	\end{tikzpicture}
	\caption{A geodesic segment, $v\in T_y \Sigma$, and $\Rcap v \in T_x \Sigma$.}
\end{figure}
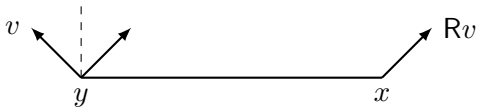

The reflection map $\Rcap$ used to define the scattering energy also leads to a direct proof of a sharp Sobolev inequality in surfaces.  Indeed, if $f$ is a smooth, compactly supported function on $\R^2$, the Cauchy-Schwarz inequality implies
\begin{align*}
\left( \int_{\R^2} |\nabla f|dA \right)^2 \geq \iint_{\R^2 \times \R^2} \langle \nabla_x f , \Rcap \nabla_y f\rangle \,  dA_x dA_y = 4\pi \int_{\R^2} f^2 \, dA.  
\end{align*}
The equality follows by direct calculation---see the proof of Theorem \ref{TSob}---and establishes the Sobolev inequality with the sharp constant.  A similar, more general Sobolev inequality holds in surfaces which satisfy the hypotheses of Theorem \ref{Ciso}.  We state and prove this result in Theorem \ref{TSob}.  Peter Topping has also described to us an elegant direct proof of the Sobolev inequality in the plane, which we outline after the proof of Theorem \ref{TSob}. 

The scattering energy of a domain $\Omega$ depends on the geodesics joining pairs of points $x,y \in \partial \Omega$ in $\Sigma$.  It is therefore intrinsic to $\Omega$ when $\Omega$ is geodesically convex---in this important special case, we prove the following curvature-free isoperimetric inequality:  

\begin{theorem}
\label{Themisphere}	
Suppose that the closure $\overline{\Omega}$ of $\Omega$ is strictly convex, in that all geodesics between pairs of points $x,y \in \overline{\Omega}$ remain in the interior of $\Omega$ except possibly at their endpoints.  Then
	\begin{equation}
	\label{Themi-eqn-2}
	L^{2} - 2\pi A > \Ecal(\Omega).  
	\end{equation}	
\end{theorem}

Theorem \ref{Themisphere} is related to several isoperimetric inequalities of Croke \cite{CrokeEigen, Croke-hemisphere}.  Theorem 11 in \cite{CrokeEigen} implies that any domain $\Omega$ as in Theorem \ref{Themisphere} satisfies the isoperimetric inequality $L^{2} \geq 2\pi A$.  In fact, Croke's result applies to any bounded domain $\Omega$ in a surface $\Sigma$ satisfying the hypotheses of Theorem \ref{Ciso}, whether or not $\Omega$ satisfies the convexity hypothesis in Theorem \ref{Themisphere}.  Equality in Croke's result holds only for domains isometric to round hemispheres.  Although closed round hemispheres do not satisfy the hypotheses of Theorems \ref{Ciso} and \ref{Themisphere}, by considering subdomains of the hemisphere, one can see that $L^{2} > 2\pi A$ is the strongest curvature-independent isoperimetric inequality which holds under these conditions, and that Theorem \ref{Themisphere} is optimal.  The proof of Theorem \ref{Themisphere} implicitly gives a formula for $L^{2} - 2\pi A$ for strictly convex domains.  In \cite{Croke-hemisphere}, Croke gives a formula for $L^{2} - 2\pi A$, which is valid under the hypotheses of Theorems \ref{Ciso} and \ref{Themisphere}, related to the proportion of maximal geodesic segments in $\Omega$ which intersect.  It would be interesting to explore the relationship between the results in \cite{Croke-hemisphere} and Theorems \ref{Ciso} and \ref{Themisphere} in greater depth.

The proof of Theorem \ref{Trigidity}, which characterizes domains $\Omega$ with $\Ecal(\Omega) = 0$, draws on the beautiful theory of boundary and scattering rigidity, in particular, the work of Pestov-Uhlmann \cite{Uhlmann}.   Indeed, the condition $\Ecal(\Omega) = 0$ implies that $\Omega$ is strictly convex, and moreover, that the boundary distance function is invariant under length-preserving maps of the boundary.  Boundary distance rigidity for simple Riemannian surfaces \cite[Theorem 1.1]{Uhlmann} then implies that each such map of $\partial \Omega$ extends to an isometry of $\Omega$, which implies that $\Omega$ is rotationally symmetric.  In fact, an explicit construction of an $O(2)$-invariant metric by Arcostanzo-Michel \cite{Michelrot} which has the same boundary distance function as $\Omega$ gives a complete description of $\Omega$ when $\Ecal(\Omega) = 0$.  The information used to calculate $\Ecal(\Omega)$ is closely related to the {\em scattering map} studied by Wen \cite{Wen}, and one can also use scattering rigidity for simple Riemannian surfaces \cite[Theorem 1.7]{Wen} to deduce the rotational symmetry of a domain with $\Ecal(\Omega)=0$. 

Efforts to understand the stability of isoperimetric inequalities go back at least to the work of Bonnesen \cite{Bonnesen} on the isoperimetric inequality in the plane.  A basic idea of such work is to show that the isoperimetric deficit bounds a nonnegative quantity measuring the asymmetry of the domain, so that a domain with small isoperimetric deficit is then close to being symmetric in a quantitative way.  The definition of $\Ecal(\Omega)$ generalizes such a deficit quantity defined in \cite{McGrath} for domains in constant curvature model surfaces, which extended work of Pleijel \cite{Pleijel} giving a formula for the isoperimetric deficit of a convex plane domain.  We note that Banchoff-Pohl  have \cite{Banchoffpohl} also generalized  Pleijel's result, in a different direction than the results of \cite{McGrath} and this paper.
 We also note that when $\Sigma$ has constant curvature, the definition of $\Rcap \nu_y$ coincides with the rotation by $\frac{\pi}{2}$ of the vector field $V(x,y)$ in H\'{e}lein's proof of the isoperimetric inequality \cite{Helein}. 

Since the breakthrough work \cite{Fusco}, there has been remarkable progress in understanding the stability of the Euclidean isoperimetric inequality.  Recent study has turned to the stability of isoperimetric inequalities in Riemannian manifolds \cite{BogeleinHyperbolic, Fusco-Sph, Chodosh}, and it has been shown \cite{Chodosh} that a natural generalization of the quantitative isoperimetric inequality in \cite{Fusco} does not hold in certain Riemannian manifolds.  In Euclidean space \cite{Fusco} and spaces of constant curvature \cite{Fusco-Sph, BogeleinHyperbolic}, these results quantify the asymmetry of a domain by comparison to a canonical symmetric domain, a ball, by an invariant known as the {\em Fraenkel asymmetry}.  The scattering energy quantifies the asymmetry of a domain $\Omega$ entirely in terms of the geometry of $\Omega$ and its geodesic hull.  

In Theorem \ref{Theuc}, we briefly study the scattering energy of higher-dimensional domains.  

\subsection*{Notation and Conventions} 
$\phantom{ab}$
\nopagebreak

Given $x,y \in \Sigma$, we let $r$ denote the distance between $x$ and $y$, and we denote differentiation with respect to $x$ or $y$ by appropriate subscripts.

For $x, y \in \Sigma$, we write $\sqrt{g}(x, y)$ to denote the function giving the volume element in normal coordinates about $x$ and evaluated at $y$.  Thus, the metric in these coordinates is $dr^{2} + \sqrt{g}^{2}(x,y) d\theta^{2}$.  By \cite[Lemma 5]{Yau}, $\sqrt{g}(x,y) = \sqrt{g}(y,x)$.  For this reason, we simply write $\sqrt{g}$ when there is no risk of confusion. 

\subsection*{Acknowledgements}
$\phantom{ab}$
\nopagebreak

We are very happy to thank Christopher Croke, Joseph H.G. Fu and Robert Kusner for their interest in this work and several helpful discussions.  We thank Peter Topping for sharing with us a complex analytic proof of the sharp Euclidean Sobolev inequality (see Section \ref{SSob}).


\section{Proof of Theorem \ref{Ciso}}
\label{dom-in-conv}

\begin{lemma}
\label{Lcalc}
Given distinct $x, y \in \Sigma$, define $w(y;x) := 1/\sqrt{g} + \Delta_x r$.  Then
\begin{enumerate}[label=\emph{(\roman*)}]
\item
\label{Lcalc-i} 
For any $v_y \in T_y \Sigma$, $\ddiv_x ( \Rcap v_y ) =  w(y;x)\langle \nabla_y r, v_y \rangle$.   
\item 
\label{Lcalc-ii} 
$\ddiv_y(w(y;x) \nabla_y r) = 4\pi \delta_x - 1/ \sqrt{g}^2 + \Delta_x r \Delta_y r$, where $\delta_x$ is the Dirac delta distribution centered at $x$. 
\item
\label{Lcalc-iii} Let $\Gamma$ be the geodesic segment joining $x$ and $y$.  Then 
\begin{align*}
\min_\Gamma K \leq \frac{1}{\sqrt{g}^2} - \Delta_x r \Delta_y r \leq \max_\Gamma K,
\end{align*}
where $K$ is the curvature of $\Sigma$, and if $\max_\Gamma K>  0$, we assume that $r(x, y)< \pi /  (2\sqrt{\max_\Gamma K})$.  Equality holds in each inequality if and only if $K$ is constant on $\Gamma$.
\end{enumerate}
\end{lemma}

\begin{proof}

	For {\em \ref{Lcalc-i}}, we may assume without loss of generality that $|v_y| = 1$.  Now take polar coordinates $(r, \theta)$ centered at $y$ in which $\langle v_y, \nabla_y r\rangle  = \cos \theta$.  In these coordinates, 
\begin{align*}
\Rcap v_y = \cos \theta \frac{\partial}{\partial r} + \frac{\sin \theta}{\sqrt{g}} \frac{\partial}{\partial \theta}.
\end{align*}
We then calculate
\begin{equation}
\label{Lcalc-i-eqn}
\ddiv_x \Rcap v_y = \frac{\sqrt{g}_{r}  +1}{\sqrt{g}} \cos \theta = \frac{ \sqrt{g}_r+1}{\sqrt{g}} \langle \nabla_y r, v_y \rangle.
\end{equation}

	To establish {\em \ref{Lcalc-ii}}, first suppose that $x\neq y$.  Since $\sqrt{g}(x, y) = \sqrt{g}(y, x)$, a calculation in polar coordinates about $x$ reveals that $\ddiv_y( \nabla_y r/ \sqrt{g}) = 0$, so it remains to compute the divergence of $\Delta_x r \nabla_y r$.  We have
\begin{align*}
\ddiv_y( \Delta_x r \nabla_y r) &= \Delta_x r \,  \ddiv_y( \nabla_y r) + \langle \nabla_y \Delta_x r , \nabla_y r \rangle\\ 
& = \Delta_x r \Delta_y r + \langle \nabla_y \Delta_x r, \nabla_y r\rangle.
\end{align*}

	Consider the geodesic segment $\Gamma: [0, r]\rightarrow \Sigma$ with $\Gamma(0) = x$ and $\Gamma(r) = y$.  For any $0\leq l_{1} < l_{2} \leq r$, we define $u(l_{1},l_{2}) := \sqrt{g}(\Gamma(l_{1}), \Gamma(l_{2}))$.  In this notation, $\Delta_x r = - \frac{\partial \log u}{\partial l_{1}}(0, r)$ and $\Delta_y r = \frac{\partial \log u}{\partial l_{2}}(0, r)$.  Moreover, 
\begin{align*}
\langle \nabla_y \Delta_x r, \nabla_y r\rangle =-  \frac{\partial^2 \log u }{\partial l_{1} \partial l_{2}}(0, r).
\end{align*}
	For fixed $s \in [0,r]$, $u(s,t)$ solves the Jacobi equation $u_{tt} + K u =0$, with initial values $u(s, s) = 0$ and $\lim_{t \searrow s} u_t(s, t) = 1$.  Therefore, 
\begin{align*}
u(s, t) = u(0, s)u(0, t) \int_{s}^t \frac{1}{u^2(0, l)}dl.
\end{align*}
	From this it follows that
\begin{align*}
\frac{ \partial \log u}{\partial l_{2}}(s, r)= \Delta_y r + \frac{u(0, s)}{u(0, r) u(s,r)}, 
\end{align*}
where $\Delta_y r = \frac{\partial \log u}{\partial l_{2}}(0,r)$ is independent of $s \in [0,r]$.  We then calculate
\begin{align*}
-\langle \nabla_y \Delta_x r, \nabla_y r\rangle = \left. \frac{\partial}{\partial s}\right|_{s= 0} \left(  \Delta_y r +\frac{u(0, s)}{u(0, r) u(s,r)}\right)= \frac{1}{u(0, r)^2} = \frac{1}{\sqrt{g}^2}.
\end{align*}
	This completes the calculation of $\ddiv_y(w(y;x) \nabla_y r)$ when $x\neq y$.  To justify the distributional term $4\pi \delta_x$, we must prove that
\begin{equation}
\label{4-pi-delta}
\lim_{\varepsilon \searrow 0} \int_{\partial B_{\varepsilon}(x)} w(y;x)\,  ds_y =  4\pi.
\end{equation}
This follows by direct calculation using the formula for $w(y;x)$, using that $ds_y = \sqrt{g}\,  d\theta$ in polar coordinates about $x$.

It is helpful to note that $\ddiv_y( w(y;x) \nabla_y r)$ is symmetric in $x$ and $y$, even though $w(y;x)$ generally is not.

To establish {\em \ref{Lcalc-iii}}, we let $\sn_{\Kcal}(r)$ be the length of a Jacobi field $J(r)$ with initial conditions $J(0) = 0$, $|J'(0)| = 1$ in the model surface with constant curvature $\mathcal{K}$.  Thus, $\sn_{\Kcal}(r) = \frac{1}{\sqrt{\Kcal}} \sin \sqrt{\Kcal} r$ if $\mathcal{K} > 0$, $\sn_{\Kcal}(r) = r$ if $\mathcal{K} = 0$ and $\sn_{\Kcal}(r) = \frac{1}{\sqrt{|\Kcal|}} \sinh \sqrt{|\Kcal|} r$ if $\mathcal{K} < 0$.  We let $\cs_{\Kcal}(r) = \sn_{\Kcal}'(r)$ and $\ct_{\Kcal}(r) = \frac{\cs_{\Kcal}(r)}{\sn_{\Kcal}(r)}$.  To simplify notation, let $\Kmin = \min_\Gamma K$ and $\Kmax = \max_\Gamma K$. 
	
By the Rauch and Laplacian comparison theorems, 
\begin{equation}
\label{Ecomp}
\begin{gathered}
\sn_{\Kmax}(r) \leq  \sqrt{g} \ \leq \sn_{\Kmin}(r), \\
 \ct^2_{\Kmax}(r) \leq  \Delta_x r \Delta_y r  \leq \ct^2_{\Kmin}(r).  
 \end{gathered}
\end{equation}
	We then estimate 
\begin{align*}
\Kmin  = \frac{1}{\sn^2_{\Kmin}(r)} - \ct_{\Kmin}^{2}(r) 
 \leq  \frac{1}{\sqrt{g}^2} - \Delta_x r \Delta_y r 
 \leq  \frac{1}{\sn_{\Kmax}^{2}(r)} - \ct_{\Kmax}^{2}(r)  =  \Kmax.
\end{align*} 
	Each inequality is an equality if and only if $K$ is constant along $\Gamma$.
\end{proof}

\begin{prop}
\label{Tdeficit}
For $\Omega$ as above,
\begin{align}
\label{EA}
\Ecal(\Omega) = L^2 - 4\pi A + \iint_{\Omega\times \Omega} \frac{1}{\sqrt{g}^2(x,y)} - \Delta_x r \Delta_y r \, dA_x dA_y,
\end{align}
where $A$ is the area of $\Omega$ and $L$ is the boundary length. 
\end{prop}

\begin{proof}
 $\Rcap \nu_y$ has unit length, so $\frac{1}{2}|\nu_x - \Rcap\nu_y |^2 = 1 - \langle \Rcap \nu_y , \nu_x\rangle$.  Integrating and using the divergence theorem with Lemma \ref{Lcalc}(i) gives
\begin{align*}
\frac{1}{2}\int_{x\in \partial \Omega}| \nu_x- \Rcap \nu_y |^2 \, ds_x  
&= L - \int_{x\in \Omega} \ddiv_x \left( \Rcap \nu_y \right) dA_x \\
&= L - \int_{x\in \Omega}\left \langle w(y;x) \nabla_y r, \nu_y \right \rangle dA_x.
\end{align*}
Integrating again, using Fubini's theorem and the divergence theorem, 
\begin{align*}
\Ecal(\Omega) = L^2 -  \int_{x\in \Omega} \int_{y\in  \Omega} \ddiv_y\left( w(y;x)\nabla_y r\right) dA_y dA_x.
\end{align*} 
Working Lemma \ref{Lcalc}(ii) into the above establishes \eqref{EA}.
\end{proof}

\begin{proof}[Proof of Theorem \ref{Ciso}]
Using Lemma \ref{Lcalc}(iii) and \eqref{EA}, we have
\begin{align*}
L^2 - 4\pi A +\iint_{\Omega\times \Omega} \max_{\Gamma(x,y)} K \, dA_x dA_y \geq \Ecal(\Omega),
\end{align*}
which implies \eqref{Eiso}.  Equality requires equality in Lemma \ref{Lcalc}(iii) along all geodesic segments $\Gamma(x,y)$, for all $x,y \in \Omega$.  This requires that $\Sigma $ have constant curvature along all such geodesic segments.  Since any two segments $\Gamma(x,y)$ and $\Gamma(z,w)$ can be joined by a third such segment $\Gamma(x,z)$, this implies that $\Omega'$ has constant curvature. \end{proof}


\section{Proof of the Sobolev Inequality}
\label{SSob}
\nopagebreak

In this section, we use Lemma \ref{Lcalc} and an analysis of the reflection map $\Rcap$ to prove a sharp Sobolev inequality. 

\begin{theorem}
	\label{TSob}
	Let $f$ be a compactly supported, smooth function on $\Sigma$, and if $\sup_{\Sigma} K> 0$, suppose $\emph{diam}(\emph{supp}(f)) \leq \pi/ (2\sqrt{\sup_{\Sigma}K})$.  If $\partial \Sigma \neq \emptyset$, we assume $\emph{supp}(f) \cap \partial \Sigma = \emptyset$.  Then
	\begin{align*}
	4\pi ||f||_{L^{2}}^{2} \leq || \nabla f ||_{L^{1}}^{2} + ( \sup_{\Sigma} K) ||f||_{L^{1}}^{2}. 
	\end{align*}
\end{theorem} 
\begin{proof}
	By smoothly approximating $|f|$, it is enough to prove the inequality for non-negative functions.  By the Cauchy-Schwarz inequality, Fubini's theorem, the divergence theorem, and Lemma \ref{Lcalc}, which implies in particular that $\ddiv_y(w(y; x) \nabla_y r) \geq 4\pi \delta_x - \sup_\Sigma K$, we then have
	\begin{align*}
	\left( \int_{\Sigma} |\nabla f|dA \right)^2 &\geq \int_{\Sigma} \int_{\Sigma} \langle \nabla_x f , \Rcap \nabla_y f\rangle \,  dA_x dA_y\\
	&= - \int_{\Sigma} \int_{\Sigma} f(x)\,  \ddiv_x ( \Rcap \nabla_y f ) \, dA_xdA_y\\
	&= - \int_{\Sigma} f(x) \int_{\Sigma} \langle w(y;x)  \nabla_y r, \nabla_y f \rangle \, dA_y dA_x\\
	&= \int_{\Sigma} f(x) \int_{\Sigma} f(y) \, \ddiv_y( w(y; x) \nabla_y r)\,  dA_y dA_x\\
	&\geq 4\pi \int_{\Sigma} f^2 \, dA  -  (\sup_\Sigma K) \left( \int_{\Sigma} f dA\right)^2.
	\end{align*}
\end{proof}

By smoothly approximating the characteristic function of a compact domain, one can infer the sharp isoperimetric inequality $L^{2} - 4\pi A + \Kmax A^{2} \geq 0$ for surfaces with $K \leq \Kmax$ as a corollary of Theorem \ref{TSob}, albeit without the lower bound in terms of $\Ecal(\Omega)$ in Theorem \ref{Ciso}.  Conversely, Howard has shown \cite{Howard} that any isoperimetric inequality on a Riemannian surface of the form $L^{2} - c_1 A + c_2 A^{2} \geq 0$ implies a corresponding Sobolev inequality $|| \nabla f ||_{L^{1}}^{2} - c_1 ||f||_{L^{2}}^{2} + c_2 ||f||_{L^{1}}^{2} \geq 0$ for compactly supported functions of bounded variation.  

We conclude this section with another proof, pointed out to us by Peter Topping (see also the proof of \cite[Lemma 1]{ToppingWente}), of the special case of Theorem \ref{TSob} when $\Sigma$ is the Euclidean plane.  Identify $\R^2$ with $\C$ and consider the usual complex coordinate $z = x+iy$, which satisfies $dx \wedge dy = \frac{i}{2} dz\wedge d\bar{z}$.  Let $\partial_{z} = \frac{1}{2}(\partial_{x} - i\partial_{y})$ and $\partial_{\bar{z}} = \frac{1}{2}(\partial_{x} + i\partial_{y})$.  For a compactly supported smooth function $f$ on $\C$, $f_{\bar{z}} = \frac{1}{2}(f_x + i f_y)$.  Therefore, $| \nabla f | = 2| f_{\bar{z}} |$.  The Cauchy integral formula says 
\begin{align*}
f(w) = \frac{1}{2\pi i } \int_{\C} \frac{f_{\bar{z}}(z)}{z-w} dz \wedge d\bar z. 
\end{align*}
We then have

\begin{align*}
4\pi \int_{\C} f^2 dA &= 2\pi i \int_\C f^2(w) dw \wedge d \bar{w}\\
&= \int_{\C} f(w) \left( \int_{\C} \frac{f_{\bar z}(z)}{z-w} dz \wedge d\bar{z}\right) dw \wedge d\bar{w}\\
&= - \int_\C f_{\bar{z}}(z) \left( \int_\C f(w) \left( \frac{\bar{z}- \bar{w}}{z-w}\right)_{\bar w} dw \wedge d\bar{w}\right) dz \wedge d\bar{z}\\
&= \int_\C \int_\C f_{\bar{z}}(z) f_{\bar{w}}(w) \left( \frac{\bar{z}-\bar{w}}{z-w}\right) dw \wedge d\bar{w} \, dz \wedge d\bar{z}\\
&\leq \left( \int_{\C} |\nabla f| dA\right)^2, 
\end{align*}
where third equality uses that $\left( \frac{\bar{z} - \bar{w}}{z-w}\right)_{\bar{w}} = - \frac{1}{z-w}$ and Fubini's theorem, and in the fourth equality, we have integrated by parts (recall the identity $d\phi = \phi_{w} \ dw + \phi_{\bar{w}} \ d \bar{w}$ for the exterior derivative).


\section{Proof of Theorem \ref{Themisphere}}
\label{conv-dom}
\nopagebreak

In this section, we give a proof of Theorem \ref{Themisphere} in the spirit of Croke's arguments in \cite{CrokeEigen, Croke, Croke-hemisphere}.  For this, we assume $\Omega$ is strictly geodesically convex, in that for all $x,y \in \Omega$, the geodesic segment between $x$ and $y$ lies in the interior of $\Omega$, except possibly at its endpoints.  The surface $\Sigma$ can therefore be taken to coincide with $\Omega$.  

\begin{proof}[Proof of Theorem \ref{Themisphere}]
	
For $y \in \partial \Omega$ we introduce polar coordinates for $\Omega$ centered at $y$, as in Section \ref{dom-in-conv}.  In these coordinates, $\theta$ is the angle between the inward normal to $\partial \Omega$ at $y$ and the geodesic segment from $y$ to the point in question.  We let $\beta(\theta)$ be the point at which this geodesic segment again meets $\partial \Omega$, and we let $\rho(\theta) = r(y, \beta(\theta))$ be the length of this segment.  We let $\tau(\theta)$ be the speed of the path along $\partial \Omega$ parametrized by $\theta$.  We then have
\begin{equation}
\label{SE-Santalo}
\Ecal(\Omega) = L^{2} - \int_{\partial \Omega} \int_{-\frac{\pi}{2}}^{\frac{\pi}{2}} \langle \Rcap \nu_y,\nu_{\beta} \rangle \tau \ d\theta ds_y.
\end{equation}
Since $\Rcap \nu_y = \cos \theta \frac{\partial}{\partial r} + \frac{\sin \theta}{\sqrt{g}} \frac{\partial}{\partial \theta}$ and  $\nu_{\beta} = \frac{1}{\tau}\left( \sqrt{g} \frac{\partial}{\partial r} - \frac{\rho'(\theta)}{\sqrt{g}} \frac{\partial}{\partial \theta} \right)$, we have	
\begin{equation}
\label{SE-Santalo-formula}
\langle \Rcap \nu_y, \nu_{\beta} \rangle = \frac{1}{\tau}\left(\sqrt{g} \cos \theta - \rho'(\theta) \sin \theta \right).  
\end{equation}
Substituting this into (\ref{SE-Santalo}) gives
\begin{equation}
\label{SE-Santalo-int}
\Ecal(\Omega)= L^{2} + \int_{\partial \Omega} \int_{-\frac{\pi}{2}}^{\frac{\pi}{2}} \rho'(\theta) \sin \theta \ d\theta ds_y - \int_{\partial \Omega} \int_{-\frac{\pi}{2}}^{\frac{\pi}{2}} \sqrt{g} \cos \theta \ d\theta ds_y.  
\end{equation}
Santal\'o's formula (cf. \cite{CrokeEigen, Croke, Croke-hemisphere} and the sources therein) implies that	
\begin{equation}
\label{Santalo-eqn}
\int_{\partial \Omega} \int_{-\frac{\pi}{2}}^{\frac{\pi}{2}} \rho(\theta) \cos \theta \ d\theta ds_y = 2\pi A,
\end{equation}	
where $2\pi$ is the length of the unit tangent fibre at each point of $\Omega$.  Integrating by parts in the inner integral in the first integral expression in (\ref{SE-Santalo-int}), and noting that $\rho(-\frac{\pi}{2}) = \rho(\frac{\pi}{2}) = 0$, we have
\begin{equation}
\label{Santalo-aux}
\Ecal(\Omega) = L^{2} - 2\pi A - \int_{\partial \Omega} \int_{-\frac{\pi}{2}}^{\frac{\pi}{2}} \sqrt{g} \cos \theta \ d\theta ds_y,
\end{equation}	
which implies Theorem \ref{Themisphere}. 
\end{proof} 

The coordinate vector field $\frac{\partial}{\partial \theta}$ in the proof of Theorem \ref{Themisphere} is a Jacobi field along radial geodesic segments from $y$, and $\sqrt{g}(y,x)$ is the length of $\frac{\partial}{\partial \theta}$ at $x$.  The integral expression $\int_{\partial \Omega} \int_{-\frac{\pi}{2}}^{\frac{\pi}{2}} \sqrt{g} \cos \theta \ d\theta ds_y$ in (\ref{Santalo-aux}) therefore has a natural interpretation in terms of Jacobi fields and gives, together with $\Ecal(\Omega)$, a geometric description of the isoperimetric deficit $L^{2} - 2\pi A$.  

\begin{remark}
\label{Alt-hemi-exp}
Using as in the proof of \ref{Lcalc}(ii) that $\ddiv_y(\nabla_y r / \sqrt{g}) = 2\pi \delta_{x}$, a calculation similar to the one in the proof of Proposition \ref{Tdeficit} establishes that for any domain $\Omega$ as in Theorem \ref{Ciso}, 
\begin{align*}
\Ecal(\Omega) = L^2 - 2\pi A - \int_{y\in\partial \Omega} \int_{x\in \Omega} \Delta_x r \langle \nabla_y r, \nu_y\rangle dA_x ds_y. 
\end{align*}
It is also possible to prove \ref{Themisphere} using this: one integrates $\Delta_x r$ along the geodesics based at $y \in \partial \Omega$ as they sweep out $\Omega$, using the expression $\sqrt{g}_r / \sqrt{g}$ for $\Delta_x r$ in normal coordinates about $y$ as above.  In these coordinates, $\cos \theta = \langle \nabla_y r, \nu_y\rangle$, which is positive because $\Omega$ is convex. 
\end{remark}


\section{Domains with Vanishing Asymmetry}
\label{symmetry}
\nopagebreak

In this section, we prove Theorem \ref{Trigidity}: that domains with vanishing scattering energy are convex and $O(2)$-symmetric.  We note that a convex, $SO(2)$-invariant metric is also $O(2)$-invariant---one can see this by writing such a metric in an appropriate polar coordinate system.  

We adopt the following notation: given $x,y$ in the same component of $\partial \Omega$, we write $r^{\partial \Omega}(x, y)$ for their distance in $(\partial \Omega, g)$, i.e. the minimum length of a boundary arc between $x$ and $y$.  We continue to write $r(x, y)$ for their distance in $\Sigma$ and call the restriction of $r(x,y)$ to $\partial \Omega \times \partial \Omega$ the \emph{boundary distance function}.

\begin{lemma}
\label{Lasymzero}
Suppose $\Ecal(\Omega) = 0$.  	
\begin{enumerate}[label=\emph{(\roman*)}]
\item 
\label{Lasymzero-i}
$\Omega$ is geodesically convex---in particular, $\partial \Omega$ is connected. 
\item 
\label{Lasymzero-ii}
For all distinct $x, y\in \partial \Omega$, 

\begin{enumerate}[label=\emph{(\alph*)}]
\item 
\label{Lasym-ii-a}
The angles at $x$ and $y$ between $\partial \Omega$ and the geodesic chord joining $x$ and $y$ are equal. 

\item 
\label{Lasym-ii-b}
$r(x, y)$ depends only on $r^{\partial \Omega}(x,y)$. 

\item 
\label{Lasym-ii-c}
The angle in \ref{Lasym-ii-a} depends only on $r^{\partial \Omega}(x,y)$. 
\end{enumerate}

\item 
\label{Lasymzero-iii}
The geodesic curvature $\kappa$ of $\partial \Omega$ is a positive constant---in particular, $\Omega$ is strictly geodesically convex.    
\end{enumerate}
\end{lemma}

\begin{proof}

To establish {\em \ref{Lasymzero-i}}, suppose that $\Omega$ is not convex.  Then we can find a geodesic segment $\Gamma$ which has nonempty intersection with $\Omega$ and passes through (at least) three points $x, y, z \in \partial \Omega$.   By a general position argument, we may further suppose that $\Gamma$ is not tangent to $\partial \Omega$ at any of $x, y$ or $z$. 

By the definition of $\Rcap$, the oriented angle between $\dot{\Gamma}$ and $\nu_p$ is $\pi$ minus the angle between $\dot{\Gamma}$ and $\Rcap \nu_p$, where $p$ is any of $x$, $y$, and $z$ and $\Rcap$ is calculated at any other of $x$, $y$ and $z$.  But then the equations $\Rcap \nu_x = \nu_y, \Rcap \nu_y = \nu_z$, and $\Rcap \nu_z = \nu_x$, which hold because $\Ecal(\Omega) = 0$, together imply that $\Gamma$ is tangent to $\partial \Omega$ at $x$, $y$, and $z$.  This contradiction proves {\em \ref{Lasymzero-i}}. 

{\em \ref{Lasymzero-ii}\ref{Lasym-ii-a}} is just a restatement of the condition that $\Rcap \nu_y = \nu_x$. 
 
Now let $\sigma : \R\rightarrow \partial \Omega$ be a covering map which parametrizes $\partial \Omega$ by arclength and $T$ the smooth unit tangent field to $\partial \Omega$ satisfying $T_\sigma = \dot{\sigma}$.  Define a map $\Theta: \R^2 \times [0, 1] \rightarrow \Omega$, where $\Theta(a, b, \cdot)$ parametrizes with constant speed the geodesic segment between $\sigma(a)$ and $\sigma(b)$ when $\sigma(a)\neq \sigma(b)$ and $\Theta(a, b, \cdot) := \sigma(a)$ when $\sigma(a) = \sigma(b)$.   For a geodesic chord $\Theta(a, b, \cdot)$, let $\eta$ be the unit outward pointing conormal vector to the boundary.  By the first variation formula for arclength, for any $l \in (0, L(\partial \Omega)/2]$, 
\begin{align*}
\frac{\partial}{\partial t} L\big(\Theta(t, t+l, \cdot)\big) = \langle T_{\sigma(t)}, \eta_{\sigma(t)} \rangle + \langle T_{\sigma(t+l)}, \eta_{\sigma(t+l)} \rangle=0, 
\end{align*} 
where the second equality follows from {\em \ref{Lasym-ii-a}}.  This proves {\em \ref{Lasym-ii-b}}, since $\partial \Omega$ is connected. 

To prove {\em \ref{Lasym-ii-c}}, for $t\in \R$ and $l \in (0, L(\partial \Omega)/2]$, by the first variation formula, 
\begin{align*}
\frac{\partial }{\partial s} L\big( \Theta(t-s, t+s+l, \cdot)\big) &= - \langle T_{\sigma(t-s)},\eta_{\sigma(t-s)}\rangle + \langle T_{\sigma(t+s+l)},  \eta_{\sigma(t+s+l)}\rangle\\
& = -2\langle T_{\sigma(t-s)}, \eta_{\sigma(t-s)}\rangle
\end{align*}
for all $s$ close enough to $0$ that $\Theta(t-s , t+s+l,\cdot)$ is an immersion, where the second equality follows from {\em \ref{Lasym-ii-a}}.  By differentiating this equation with respect to $t$, switching the order of differentiation and using {\em \ref{Lasym-ii-b}}, we conclude that $\frac{\partial}{\partial t} \langle T_{\sigma(t-s)}, \eta_{\sigma(t-s)} \rangle = 0$, which implies {\em \ref{Lasym-ii-c}} since $\partial \Omega$ is connected. 

To establish {\em \ref{Lasymzero-iii}}, given $s\in \R$ and $t\in (s, s+L(\partial \Omega)/2)$, let $\theta_1(s, t)$ and $\theta_2(s, t)$ be the angles between the geodesic chord joining $\sigma(s)$ and $\sigma(t)$ and the portion of the boundary $\{ \sigma(r) : r\in (s, t)\}$ at $\sigma(s)$ and $\sigma(t)$ respectively.  In \cite[p. 72]{Willmore}, it is shown that
\begin{align}
\label{Ecurvature}
\kappa(\sigma(s)) = \lim_{t \searrow s} \frac{\theta_1(s, t) + \theta_2(s, t)}{t-s}.
\end{align}

By {\em \ref{Lasymzero-ii}\ref{Lasym-ii-a}}, $\theta_1(s,t) = \theta_2(s, t)$.  Moreover, {\em \ref{Lasymzero-ii}\ref{Lasym-ii-c}} implies that $\theta_1(s, t)$ depends only on $t-s$, so it follows from \eqref{Ecurvature} that $\kappa(\sigma(s))$ is constant.  By the convexity of $\Omega$, $\theta_1(s, t) \geq 0$, so $\kappa \geq 0$.  If $\kappa$ were zero, $\partial \Omega$ would be a closed geodesic, contradicting the geodesic convexity of $\Sigma$, so in fact $\kappa > 0$.	
\end{proof}

By Lemma \ref{Lasymzero}, $\partial \Omega$ is strictly convex and the boundary distance function is invariant under isometries of the boundary.  Arcostanzo-Michel  \cite{Michelrot} studied surfaces with the preceding properties and proved (within \cite{Michelrot}, see Lemma 2 in Section 3.2 and the Proposition in 3.3 for more details):

\begin{enumerate}[label={(\alph*)}]
\item 
\label{AM-a}
There is a rotationally symmetric metric $g_f = dr^2 + f(r)^2 d\theta^2$ on a disk $D$ which has the same boundary distance function as $\Omega$. 
\item 
\label{AM-b} 
If the metric on $\Omega$ is analytic or has nonpositive curvature, then $\Omega$ is isometric to the rotationally symmetric $(D, g_f)$. 
\end{enumerate} 

The proof of (a) in \cite{Michelrot} is constructive and explicitly writes the inverse function to $f$ as an integral defined in terms of the boundary distance function.  We are now ready to prove Theorem \ref{Trigidity}: 
\begin{proof}[Proof of Theorem \ref{Trigidity}]
We have assumed that $\Omega$ has no boundary conjugate points.  Together with strict geodesic convexity and the fact that $\partial \Omega$ has positive geodesic curvature, this implies that $\Omega$ is a \emph{simple manifold} in the terminology of Pestov-Uhlmann \cite{Uhlmann} and Wen \cite{Wen}.  Because $(\Omega, g)$ and $(D, g_f)$ have the same boundary distance function, it follows from Theorem 1.2 of \cite{Uhlmann} that $\Omega$ is isometric to $(D, g_f)$. 
\end{proof}

We note that the identity \eqref{EA} for $\Ecal(\Omega)$ used in the proof of Theorem \ref{Ciso} is similar to an integral-geometric identity used by Mukhometov \cite{Muhometov} to prove that simple, conformally Euclidean surfaces are boundary distance rigid within their pointwise-conformal class.  It would be interesting to know if there is a modification of $\Ecal(\Omega)$ which has applications to boundary or scattering rigidity.  

We conclude this section by noting that one can also establish the rotational symmetry of a domain $\Omega$ with $\Ecal(\Omega) = 0$ without drawing on the construction from \cite{Michelrot}: 

By \ref{Lasymzero}{\em \ref{Lasymzero-ii}\ref{Lasym-ii-b}} and either boundary rigidity \cite[Theorem 1.2]{Uhlmann} or scattering rigidity \cite[Theorem 1.7]{Wen} for simple surfaces, it follows that each isometry of $\partial \Omega$ extends to an isometry of $\Omega$.  Let $\sigma : \R\rightarrow \partial \Omega$ be a Riemannian covering map, as in the proof of Lemma \ref{Lasymzero}.  Given $\theta \in \R$,  define an isometry $R_\theta : \partial \Omega \rightarrow \partial \Omega$ by $R_\theta(\sigma(s)) = \sigma(s + \theta)$ for all $s\in \R$ and denote by the same symbol the extended isometry of $\Omega$. 

For $\epsilon>0$ small, fix the parallel hypersurface $C_\epsilon = \{x\in \Omega : d(x, \partial \Omega) = \epsilon\}$, $p\in C_\epsilon$, and $\theta \in \R$ so that $\theta/ L(C_\epsilon)$ is irrational.  Since $R_\theta$ restricts to an isometry of $C_\epsilon$ and $\{ R^n_\theta(p): n \in \N\}$ is dense in $C_\epsilon$, the Gaussian curvature is constant on $C_\epsilon$.  Therefore, the metric on the tubular neighborhoods $\{x \in \Omega : d(x, \partial \Omega)< \epsilon\}$ is rotationally symmetric for all small $\epsilon>0$. Finally, it is not difficult to see that the cut locus of $\partial \Omega$ is then a single point $p$ characterized by the property that $d(p, \partial \Omega) = \text{diam}(\Omega)/2$, so that $(\Omega, g)$ is rotationally symmetric.


\section{Higher Dimensions}
\label{highdim}
\nopagebreak

In this section, we briefly study the scattering energy of higher-dimensional domains, restricting our attention to subdomains of the model spaces of constant curvature.  Let $(M^n, g)$ be a domain in the hyperbolic space $\HH^n$, the Euclidean space $\R^n$ or the sphere $\Sph^n$ in which each pair of points is joined by a unique, minimizing geodesic and let $\Omega \subset M$ be a precompact domain with $C^2$ boundary. 

Given distinct $x, y \in M$, define $\Rcap : T_y M \rightarrow T_x M$ in the same way as in Section \ref{intro}.  Note then that for any $v_y \in T_y M$, we have 
\begin{align}
\label{ERr}
\Rcap v_y = \vec{v}_y +2 \langle v_y, \nabla_y r \rangle \nabla_x r,
\end{align}
where $\vec{v}_y$ is the parallel extension of $v_y$ along radial geodesics from $y$.  We write $\sqrt{g} = \sqrt{\det(g_{ij})}$ for the volume density function in normal coordinates centered at $x$ and evaluated at $y$.  Letting $f(x,y)= \sn_K r(x, y)$, where $K$ is the sectional curvature of $M$, we then have $\sqrt{g}(x,y) = f^{n-1}(x,y)$.  We write $f'$ for $\sn_K'(r(x,y))$. 

We now define a generalization of $\Ecal(\Omega)$ as follows: given $p \geq 2-n$, define 
\begin{align}
\label{Eecalh}
\Ecal_p(\Omega) := \frac{1}{2}\iint_{\partial \Omega \times \partial \Omega}\!  f^p | \nu_x- \Rcap \nu_y |^2  dx dy.
\end{align}
The case $p=0$ corresponds to $\Ecal(\Omega)$ as defined in \eqref{EAindex}.

\begin{lemma}
	\label{Ldimcalc}
	The following hold: 
	\begin{enumerate}[label=\emph{(\roman*)}]
		\item 
		\label{Ldimcalc-i}
		For distinct $x,y \in M$ and $v_y \in T_y M$, 
		\[ 
		\ddiv_x (\Rcap v_y) = (n-1) \frac{f' + 1}{f} \langle v_y, \nabla_y r\rangle.
		\] 
		\item 
		\label{Ldimcalc-ii} 
		For any $x,y \in M$,
		\begin{align*}
		\ddiv_y( f^{p-1} \nabla_y r ) =
		\begin{cases}
		|\partial B^n(1) | \delta_x &\quad p = 2-n,\\
		(n-2+p)f^{p-2} f' &\quad p>2-n,  
		\end{cases}
		\end{align*}
		where $|\partial B^n(1)|$ is the measure Euclidean unit sphere $\partial B^n(1)$.
	\end{enumerate}
\end{lemma}

\begin{proof}
	We first compute $\ddiv_x ( \vec{v}_y )$.  Take a system of spherical coordinates about $y$ in which $g = dr^2 + f^2 g_{\Sph^{n-1}}$, where $\Sph^{n-1}$ is the unit sphere in $T_y M$ and $g_{\Sph^{n-1}}$ is its round metric.  Let $v_y^{\top}$ be the vector field on $\Sph^{n-1}$ which results from the orthogonal projection of $v_y$ onto the tangent hyperplanes to $\Sph^{n-1}$.  Letting $\xi_x \in T_y M$ be the unit vector such that $exp_y \left( r(x,y) \xi_x \right) = x$, we then have $v_y = \langle v_y, \xi_x \rangle \xi_x + v_y^{\top}$.  Identifying $T_{\xi_x} \Sph^{n-1}$ with $\xi_x^{\perp} \subset T_y M$ and $T_x M$ with $T_{\xi_x} \Sph^{n-1} \oplus \text{span} \lbrace \frac{\partial}{\partial r} \rbrace$, we have 
	\[ \vec{v}_y  = \langle v_y , \xi_x \rangle \frac{\partial}{\partial r} + \frac{1}{f} v_y^{\top} = \langle \vec{v}_y , \nabla_x r \rangle \frac{\partial}{\partial r} + \frac{1}{f} v_y^{\top}. \]
	Then
	\begin{align*}
	\ddiv_x ( \vec{v}_y) &= (n-1) \frac{f'}{f} \langle \vec{v}_y, \nabla_x r\rangle + \frac{1}{f}\ddiv_{\Sph^{n-1}}( v_y^{\top})\\
	&= (n-1) \frac{f'}{f} \langle \vec{v}_y, \nabla_x r\rangle - (n-1) \frac{1}{f} \langle \vec{v}_y, \nabla_x r\rangle.
	\end{align*}
	We also have 
	\begin{align*}
	\ddiv_x\left( \langle v_y, \nabla_y r\rangle \nabla_x r\right) &= 
	\langle \nabla_x \langle \vec{v}_y, \nabla_y r \rangle , \nabla_x r \rangle + \langle \vec{v}_y, \nabla_y r \rangle \Delta_x r \\
	&= (n-1)\frac{f'}{f} \langle \vec{v}_y, \nabla_y r \rangle,
	\end{align*}
	where we have used that $\langle \vec{v}_y, \nabla_y r \rangle$ is constant along geodesics from $y$ to deduce that its gradient is orthogonal to $\nabla_x r$.  Combining the preceding with \eqref{ERr} proves {\em \ref{Ldimcalc-i}}.  
	
	The proof of {\em \ref{Ldimcalc-ii}} is a calculation in spherical coordinates about $x$. 
\end{proof}

We now state and prove the main theorem of this section, which generalizes integral identities of Banchoff-Pohl \cite{Banchoffpohl} and Gysin \cite{Gysin}.  In contrast to the two-dimensional case, these higher dimensional identities do not appear to be directly related to isoperimetry. 

\begin{theorem}
	\label{Theuc}
	Let $M$ and $\Omega$ be as above, with sectional curvature $K$. 
	\begin{enumerate}[label=\emph{(\roman*)}]
	\item
	\label{Theuc-i}	
	For $p = 2-n$, we have 
	\begin{align*}
	\Ecal_{2-n}(\Omega)= \iint_{\partial \Omega \times \partial \Omega} f^{2-n} dx dy - n |\partial B^n(1)| |\Omega| + K \iint_{\Omega \times \Omega} f^{2-n}\, dx dy.  
	\end{align*}
	\item 
	\label{Theuc-ii}	
	For $ p> 2-n$, we have 
	\begin{align*}
	\Ecal_p(\Omega) &= \iint_{\partial \Omega \times \partial \Omega} f^p\, dxdy - (n-1) (n-2+p) \iint_{\Omega\times \Omega} f' f^{p-2}\, dxdy \\ &\phantom{=} - (n-1+p)\iint_{\Omega \times \Omega} \left( (n-2+p) f^{p-2}-(n-1+p)Kf^p \, \right) dx dy.
	\end{align*}
	\end{enumerate}
\end{theorem}

\begin{proof}
	Using Lemma \ref{Ldimcalc}, for $y \in \partial \Omega$ we calculate 
	\begin{align*}
	\ddiv_x( f^p \Rcap \nu_y)  &= f^p \ddiv_x ( \Rcap \nu_y) + p f^{p-1}f' \langle \nabla_x r, \Rcap \nu_y\rangle\\
	&= f^p(n-1) \frac{f'+1}{f}\langle \nu_y , \nabla_y r \rangle + p f^{p-1}f' \langle \nu_y, \nabla_y r \rangle\\
	&= f^{p-1}((n-1+p)f' +n-1) \langle \nu_y , \nabla_y r \rangle .
	\end{align*}
	We have also
	\begin{align*}
	\ddiv_y( f^{p-1} f' \nabla_y r ) &= f' \ddiv_y (f^{p-1} \nabla_y r ) + f'' f^{p-1}\\
	&=  f' \ddiv_y (f^{p-1} \nabla_y r ) - K f^{p}.
	\end{align*}
	Using that $\frac{1}{2}|\nu_x - \Rcap \nu_y|^2 = 1 - \langle \Rcap\nu_y , \nu_x\rangle$ we have via the divergence theorem, Fubini's theorem, and the preceding that 
	\begin{align*}
	\Ecal_p(\Omega) &= \iint_{\partial \Omega \times \partial \Omega} f^p\, dx dy - \int_{x\in \Omega} \int_{y \in \partial \Omega} \ddiv_x ( f^p \Rcap \nu_y ) \,dy dx\\
	&=  \iint_{\partial \Omega \times \partial \Omega} f^p\, dx dy - (n-1) \iint_{\Omega \times \Omega} \ddiv_y ( f^{p-1} \nabla_y r ) \, dxdy \\ 
	&\phantom{=}- (n-1+p) \iint_{\Omega\times \Omega} \ddiv_y(f^{p-1} f' \nabla_y r )\, dx dy\\
	&= \iint_{\partial \Omega \times \partial \Omega} f^p \, dx dy  -(n-1)\iint_{\Omega\times \Omega} \ddiv_y(f^{p-1} \nabla_y r) \, dx dy\\
	&\phantom{=} - (n-1+p)\iint_{\Omega\times \Omega} \left( f' \ddiv_y(f^{p-1} \nabla_y r)- K f^p\, \right) dxdy  .
	\end{align*}
	When $p=2-n$, the result now follows by substituting Lemma \ref{Ldimcalc}{\em \ref{Ldimcalc-ii}} into the equation above, and noting that $f'(0) = 1$.
	When $p> 2-n$, we have by substituting from \ref{Ldimcalc}{\em \ref{Ldimcalc-ii}} that
	\begin{align*}
	\Ecal_p(\Omega) &= 
	\iint_{\partial \Omega \times \partial \Omega} f^p\, dxdy - (n-1) (n-2+p) \iint_{\Omega\times \Omega} f' f^{p-2}\, dxdy \\
	&\phantom{=} -
	(n-1+p)\iint_{\Omega \times \Omega} \left( (n-2+p)(f')^2 f^{p-2}-Kf^p \, \right) dx dy.
	\end{align*}
	The result now follows using that $(f')^2 + K f^2 = 1$. 
\end{proof}

\bibliographystyle{plain}
\bibliography{bibliography}

\end{document}